\documentclass[twoside, 12pt]{amsart}
\usepackage{latexsym}
\usepackage{amssymb,amsmath,amsopn}
\usepackage[dvips]{graphicx}   
\usepackage{color,epsfig}      
\usepackage{url}

\makeatletter
\def\@settitle{\begin{center}%
  \baselineskip14\p@\relax
  \bfseries
  \uppercasenonmath\@title
  \@title
  \ifx\@subtitle\@empty\else
     \\[1ex]\uppercasenonmath\@subtitle
     \footnotesize\mdseries\@subtitle
  \fi
  \end{center}%
}
\def\subtitle#1{\gdef\@subtitle{#1}}
\def\@subtitle{}
\makeatother

\newtheorem{theorem}{Theorem}

\newtheorem{lemma}{Lemma}

\newtheorem{defi}{Definition}
\newtheorem{problem}{Problem}

\newtheorem{conj}{Conjecture}

\DeclareMathOperator{\conv}{conv}

\DeclareMathOperator{\vol}{vol}

\newcommand{\K}{\mathcal{K}}
\newcommand{\R}{\mathbb{R}}

\renewcommand{\S}{\mathcal{S}}

\begin{document}
\title[On convex bodies that are characterizable by volume function]{On convex bodies that are characterizable by volume function}
\subtitle{"Old and recent problems for a new generation" \\ a survey}
\author[\'A. G.Horv\'ath]{\'Akos G.Horv\'ath}
\address {\'A. G.Horv\'ath \\ Department of Geometry \\ Mathematical Institute \\
Budapest University of Technology and Economics\\
H-1521 Budapest\\
Hungary}
\email{ghorvath@math.bme.hu}
\date{}


\begin{abstract}
The "old-new" concept of convex-hull function was investigated by several authors in the last seventy years. A recent research on it led to some other volume functions as the covariogram function, the widthness function or the so-called brightness functions, respectively. A very interesting fact that there are many long-standing open problems connected with these functions whose serious investigation closed before the "age of computers". In this survey, we concentrate only on the three-dimensional case, we will mention the most important concepts, statements, and problems.
\end{abstract}

\maketitle

\section{Introduction}

The "old-new" concept of convex-hull function was investigated by several authors in the last seventy years. A recent research on it led to some other volume functions as the covariogram function, the widthness function or the so-called brightness functions, respectively. A very interesting fact that there are many long-standing open problems connected with these functions whose serious investigation closed before the "age of computers". The structure of the conjectured optimal bodies reflect quite a theoretical attitude seemingly there was no computer search to support them. In this paper we collect some among them (using the necessary theoretical knowledge) to inspire the experts of the computer for such research which can reorder the map of these problems. We concentrate on the three-dimensional case, we mention the most important concepts, statements, and problems.

We use the following notation. $K$ is a convex body, the class of convex bodies is denoted by $\K$. $K+L$ means the Minkowski sum of the convex bodies, $K$, $L$ and $-K$ is the reflected image of $K$ with respect to the origin. $\R^n$ and $\S^{n-1}$ are the analytic model of the $n$-dimensional Euclidean space and the $(n-1)$-dimensional sphere, respectively. $C^1$-body, $C^2_+$-body mean convex body with continuously differentiable boundary and convex body with boundary of positive curvature, respectively. We use the following special notation
\begin{itemize}
\item $DK$: the difference body of $K$
\item $\triangle K$: the central symmetral of $K$
\item $\Pi K$: the projection body of $K$
\item $\triangledown$: the Blaschke body of $K$
\item $g_K(u)$: the covariogram function of $K$
\item $h_K(u)$: the support function of $K$
\item $w_K(u)$: the width function of $K$
\item $b_K(u)$: the brightness function of $K$
\item $G_K(u)$: the convex-hull function of $K$
\item $G(u)$: the Gauss map.
\end{itemize}
From the general theory of convex sets we use the concept of Minkowski norm generated by a centrally symmetric with respect to the origin convex body, the concept of polar body and the special $2$-dimensional norm the so-called Radon norm with its unit disk which boundary is the Radon curve. The following statements are important in our arguments:
\begin{itemize}
\item Brunn-Minkowski inequality: If $K$ and $L$ convex bodies then
$$
\vol_n\left(K+L\right)^{\frac{1}{n}}\leq \vol_n\left(K\right)^{\frac{1}{n}}+\vol_n\left(L\right)^{\frac{1}{n}}
$$
\item Alexandrov's projection theorem: Let $1\leq i\leq k\leq n-1$, and let $K$, $L$ be centrally symmetric compact convex sets, of dimension at least $i+1$, in $\R^n$. If $V_i(K|S)=V_i(L|S)$ for all $S$ $k$-dimensional subspace, then $K$ is a translate of $L$.
\item Cauchy's projection formula:
$$
V(K|u^{\perp})=\frac{1}{2}\int\limits_{\S^{n-1}}\langle u,u\rangle\mathrm{d}S(K,v).
$$
\item Minkowski's existence theorem: For the finite Borel measure $\mu$ in $\S^{n-1}$ to be $S_{n-1}(K,\cdot)$ for some convex body $K\in \K$, it is necessary and sufficient that $\mu$ not be concentrated on any great subsphere of $\S^{n-1}$, and
    $$
    \int\limits_{\S^{n-1}}u\mathrm{d}\mu(u)=0.
    $$
\item Minkowski's first inequality: Let $V^n(K,n-1;l)$ be the mixed volume \hfill \break $V^n(K,\cdots, K,L)$ where the number of the copies of $K$ is $n-1$, then
$$
V^n(K,n-1;L)\geq \vol_n^{n-1}(K)\vol_n(L),
$$
with equality if and only if $K$ and $L$ lie in parallel hyperplanes or are homothetic.
\end{itemize}

\section{The covariogram function}\label{sec:covariogram}

Let $K$ be a convex body in $\R^n$. The covariogram $g_K$ of $K$ is the function

\begin{equation}\label{def:covariagram}
g_K(t):=\vol_n(K\cap(K+t))
\end{equation}
where $t \in \R^n$ and $\vol_n $ denotes $n$-dimensional volume. This functional, which was introduced by Matheron in his book \cite{matheron} on random sets, is also sometimes called the set covariance and it coincides with the autocorrelation of the characteristic function $1_K$ of $K$:
$$
g_K = {\bf 1}_K \star {\bf 1}_{-K}.
$$

The covariogram $g_K$ is clearly unchanged by a translation or a reflection of $K$. (The term reflection will always mean reflection in a point.)
Matheron \cite{matheron} in 1986 asked the following question and conjectured a positive answer for the case $n = 2$.

\begin{problem}[Covariogram problem.] Does the covariogram determine a convex body, among all convex bodies, up to translations and reflections?
\end{problem}

The first contribution to Matheron’s question was made by Nagel \cite{nagel} in 1993, who confirmed Matheron’s conjecture for all convex polygons. Other partial results towards the complete confirmation of this conjecture in the plane have been proved by Schmitt \cite{schmitt}, Bianchi, Segala and Volcic \cite{bianchi-segala-voicic}, Bianchi \cite{bianchi}, Averkov and Bianchi \cite{averkov-bianchi}.
In general, the answer to the covariogram problem is negative, as the author \cite{bianchi} proved by finding counterexamples in $\R^n$, for any $n\geq 4$. Indeed, the covariogram of the Cartesian product of convex sets $K \subset \R^k$ and $L \subset \R^m$ is the product of the covariograms of $K$ and $L$. Thus $K\times L$ and $K \times (-L)$ have equal covariograms. However, if neither $K$ nor $L$ is centrally symmetric, then $K \times L$ is neither a translation nor a reflection of $K \times (-L)$. To satisfy these requirements the dimension of both sets must be at least two and thus the dimension of the counterexamples is at least four. We note that these counterexamples can be polytopes but not $C^1$ bodies.
For $n$-dimensional convex polytopes $P$ , Goodey, Schneider and Weil \cite{goodey-schneider-weil} prove that if $P$ is simplicial and $P$ and $-P$ are in general relative position (the polytope is a generic polytope), the covariogram determines $P$. Finally, Bianchi in \cite{bianchi-threedim} proved that for three-dimensional polytopes the conjecture is also true. We note that the general case of dimension three is open, it is still not known
whether every three-dimensional convex body is determined by its covariogram. We collected the most important statements in the following theorem:
\begin{theorem}\label{thm:covariagram}
Let $K$ be a convex body of dimension $n$ in the Euclidean $n$-space.
\begin{itemize}
\item The covariagram function determines a convex body $K$, among all convex bodies, up to translations and reflections in the following cases:
\begin{itemize}
\item $K$ is centrally symmetric,
\item $\dim K=2$,
\item $K$ is a polytope, and $\dim K=3$,
\item $K$ is a generic polytope of dimension $n$.
\end{itemize}
\item If $n\ge 4$ there are convex polytopes which don't determine by its covariagram function.
\end{itemize}
\end{theorem}

We first deal with such bodies which are connected to this nice problem.

\subsection{The difference body and the central symmetral of $K$}
\begin{figure}
\centering
\includegraphics[scale=0.8]{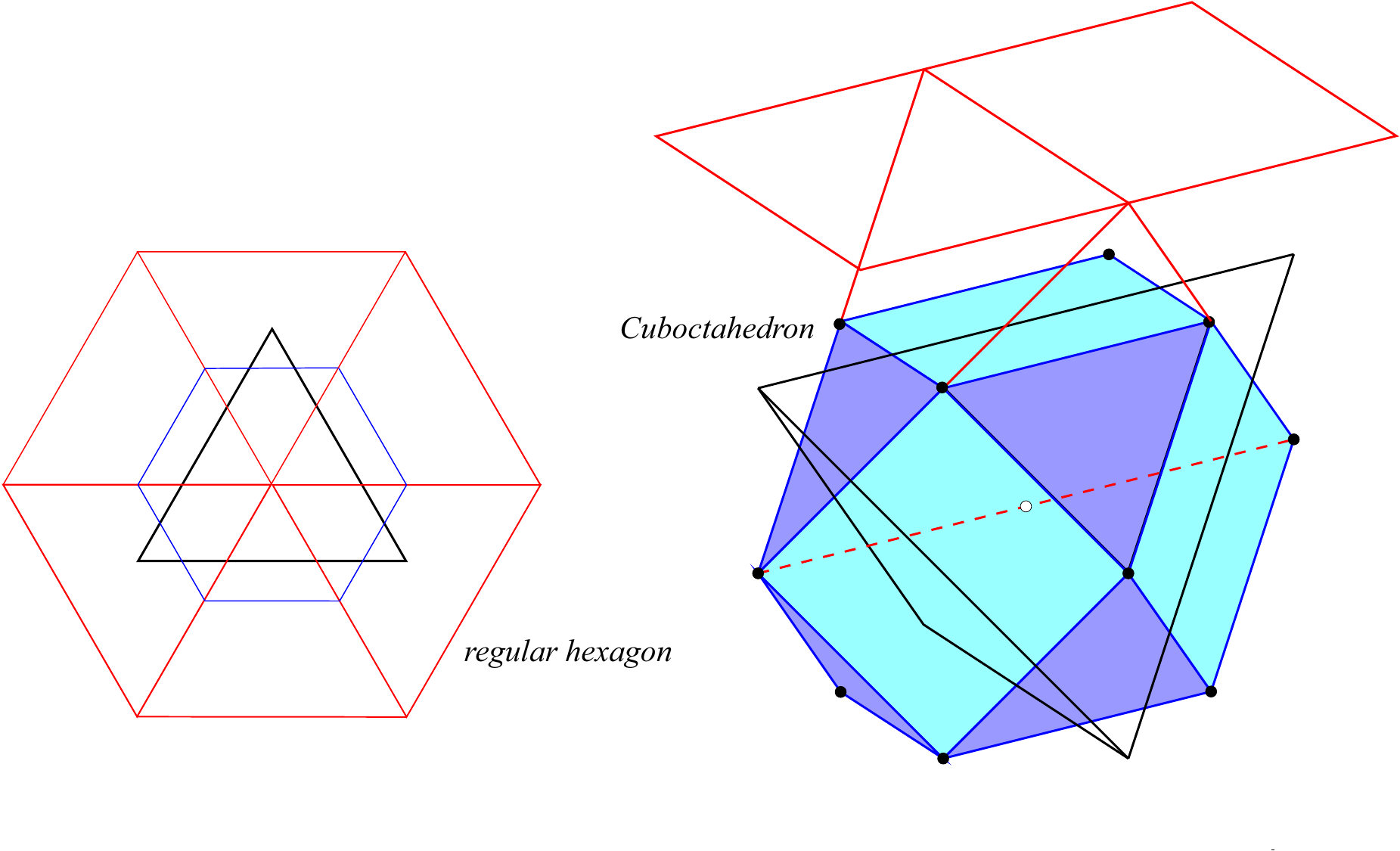}
\caption[]{The difference bodies of simplices.}
\label{fig:differencebody}
\end{figure}
First observe that the covariagram function $g_K$ determines  the volume of $K$ since $g_K(0)=\vol_n(K)$. What the covariogram
obviously does determine is its support, which is the convex body
\begin{equation}\label{def:differencebody}
\mathrm{D}K := \{x \in \R^n : K \cap (K + x)\neq \emptyset\} = K - K.
\end{equation}
This is the \emph{difference body of K}. Sometimes is more convenient to define the so-called \emph{central symmetral} $\triangle K$ of $K$ which is the homothetic by factor $1/2$ copy of the difference body. The difference body can be constructed as follows; suppose that the origin $o$ is in the interior of $K$, and take the union of all translates of $-K$ which are placed so that the corresponding translate of $o$ lies on the boundary of $K$. Finally, if we dilate by a factor $1/2$ we get $\triangle K$. The central symmetral of a regular triangle is a regular hexagon, and the central symmetral of a regular tetrahedron is the cuboctahedron (see, Figure \ref{fig:differencebody}). Since for a centrally symmetric convex body $K$ with centre $c$ holds $-K=K-2c$
$$
tK+(1-t)(-K)=K-2c(1-t),
$$
and in particular $\triangle K$ a translate of $K$. Hence a centrally symmetric convex body is determined (among all convex bodies) by its covariogram, since $2K$ is a translated copy of $\mathrm{D}K$. In general, $\triangle K$ doesn't determines $K$ meaning that there are non-congruent convex bodies with the same central symmetral.

\subsection{The support function and the width function}

The \emph{support function} $h_K$ of $K$ is defined by
\begin{equation}\label{def:supportfunction}
h_K(x)=\max\{ \langle x,y\rangle : y\in K\},
\end{equation}
for $x\in\R^n$. As a function, the support function is positively homogeneous, subadditive, that is sublinear function. It can be seen that a convex body is determined by its support function (see (0.6) in \cite{gardner}).

Let $K$ be compact, convex set in $\R^n$. Then $K$ has two supporting hyperplanes which are orthogonal to a unit vector $u$. The distance between these hyperplanes is the \emph{width} $w_K(u)$ of $K$ in the direction of $u$. The \emph{width} of $K$ is the maximal value of the width function the \emph{thickness} of $K$ is the minimal one, respectively. The formal definition of the width function is
\begin{equation}\label{def:widthfunction}
w_K(u):=h_K(u)+h_K(-u)
\end{equation}
for $u\in \S^{n-1}$. Since
$$
w_{tK+(1-t)K}=tw_K+(1-t)w_{-K}=w_K
$$
there is a whole continuum of non-congruent compact convex sets which have the same width function. One of them is the central symmetral of $K$ showing that $w_{\triangle K}=w_K$ for all convex compact body. (Generally, $\triangle K$ is the unique centred compact convex set with this property.)

The support function of the central symmetral is equal to $h_{\triangle K}=1/2(h_K+h_{-K})=1/2w_K$ thus the support function of the difference body is the width function of K. In particular, the covariogram of a convex body determines its width function, but this property also allows a great uncertainness in the determination of $K$. We can define the class of convex bodies with constant width by the equation
\begin{equation}\label{def:bodiesofconstantwidth}
w_K(u)=h_K(u)+h_K(-u)=\mathrm{constant}.
\end{equation}
Clearly, the $n$-dimensional ball is of constant width. On the plane, nonspherical examples are the Reuleaux polygons, and there is analogous theorem in any dimension. Precisely, it can be proved that nonspherical convex bodies of constant width exist in $\R^n$ for all $n\geq 2$ (see, Theorem 3.2.5 in \cite{gardner}). There are nice investigations on this class of bodies. Evidently, a convex body $K$ is of constant width if and only if $\triangle K$ is a ball. If $K$ is a centrally symmetric convex body in $\R^n$ then $h_K(u)=h_{-K}(u)=h_K(-u)$ implying that $w_K(u)=h_K(u)+h_{-K}(u)=2h_K(u)$. So if the width function of a centrally symmetric convex body is constant then its support function is also, hence the body must be a ball. Using Brunn-Minkowski inequality (see Theorem 7.1.1 in \cite{schneider}) we get that in the class of convex bodies of constant width $d$ the ball has the largest volume. In fact, the central symmetral of $K$ is $\triangle K=1/2(K+(-K))$ with the same width as of $K$ but it is centrally symmetric, hence it is a ball with the same width function as $K$. On the other hand Brunn-Minkowski inequality gives for the central symmetral (which is the ball of the same class of constant width) that
$$
\vol_n(\triangle K)^{\frac{1}{n}}=\vol_n\left(\frac{1}{2}(K+(-K))\right)^{\frac{1}{n}}\geq \vol_n(K)^{\frac{1}{n}},
$$
as we stated. Now the question that \emph{Whose body has minimal volume in a class of convex bodies of constant width $d$?} is natural.
In the class of all plane convex sets of constant width $d$, the Reuleaux triangle has least area. The first proofs of this theorem are
contained in the papers by Lebesgue \cite{lebesgue} and by Blaschke \cite{blaschke}. Similar question was investigated by P\'al \cite{pal} who
showed that the regular triangle has least area among all convex sets of given thickness $d$.
Recently, Campi, Colesanti and Gronchi investigated in the minimum volume problem in \cite{campi-colesanti-gronchi}. He considered the $3$-dimensional case and raised the following problems:
\begin{problem}[Campi-Colesanti-Gronchi]
Find a convex body of minimum volume in each of the following classes:

A) The class of convex bodies with constant width $d$;

B) The class of convex bodies with thickness $d$.

\end{problem}

Notice that the existence of solutions for each of these problems is guaranteed by standard compactness arguments. Problems A and B are still unsolved. On the other hand, several authors turned their attention to the classes of convex bodies involved in those problems. For a detailed history please read the paper \cite{campi-colesanti-gronchi}. We consider only those bodies which are good candidates to solve these problems. First we mention here the Reuleaux tetrahedron which can be get from a regular tetrahedron with edge length $d$, as the intersection of four balls centred at the vertices of the tetrahedron with radius $d$. Unfortunately, Reuleaux tetrahedron isn't a body of constant width, the midpoints of its two opposite curved edges have greater distance as $d$. Meissner and Schilling \cite{meissner-schilling} showed how to modify the Reuleaux tetrahedron to form a body of constant width, by replacing three of its edge arcs by curved patches formed as the surfaces of rotation of a circular arc. Incidentally, as Meissner mentioned in p. 49 of \cite{meissner-1911}, the ball is the only body of constant width that is bounded only by spherical pieces. According to which three edge arcs are replaced (three that have a common vertex or three that form a triangle) there result two noncongruent shapes that are called \emph{Meissner tetrahedra} (see Figure \ref{fig:meissnerbody}).

\begin{figure}
\centering
\includegraphics[scale=0.8]{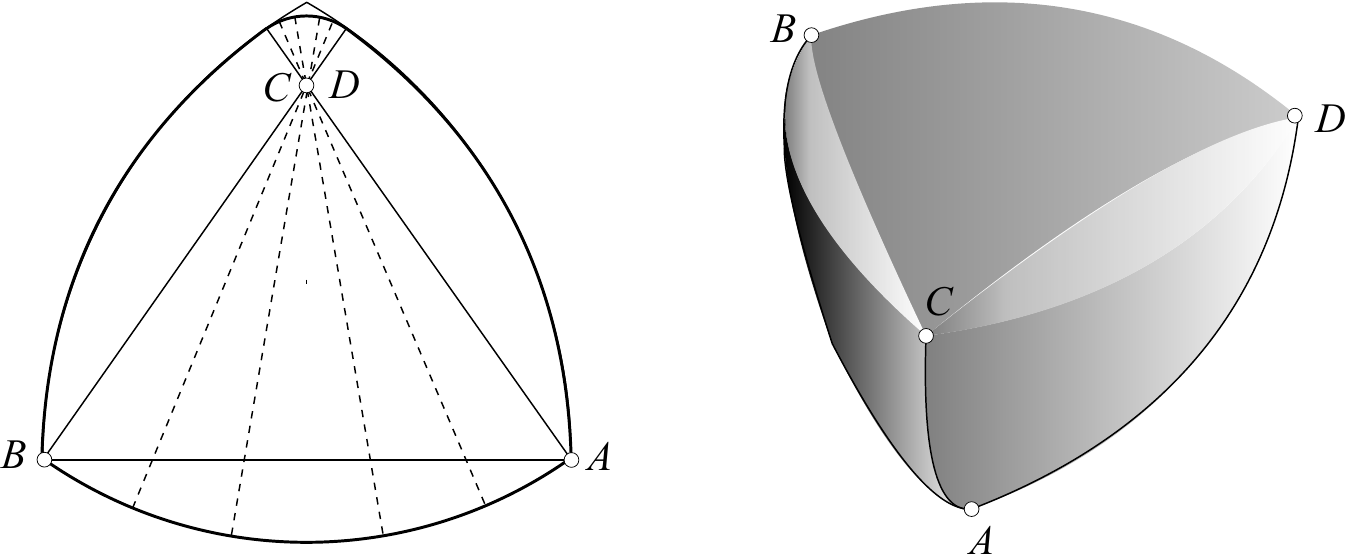}
\caption[]{The Meissner tetrahedron.}
\label{fig:meissnerbody}
\end{figure}

Bonnesen and Fenchel conjectured in \cite{bonnesen-fenchel}:

\begin{conj}[Bonnesen-Fenchel, 1934]\label{conj:bonnesenfenchel}
Meissner tetrahedra are the minimum-volume three-dimensional shapes of constant width.
\end{conj}

This conjecture is still open. (I propose to read the nice paper of Bernd and Weber \cite{kawohl-weber} on this conjecture.)  In connection with this problem, Campi, Colesanti and Gronchi showed

\begin{theorem}[Campi-Colesanti-Gronchi, 1996]\label{thm:c-c-gconstantwidth}
The minimum volume surface of revolution with constant width is the surface of revolution of a Reuleaux triangle through one of its symmetry axes.
\end{theorem}

\begin{figure}
\centering
\includegraphics[scale=0.8]{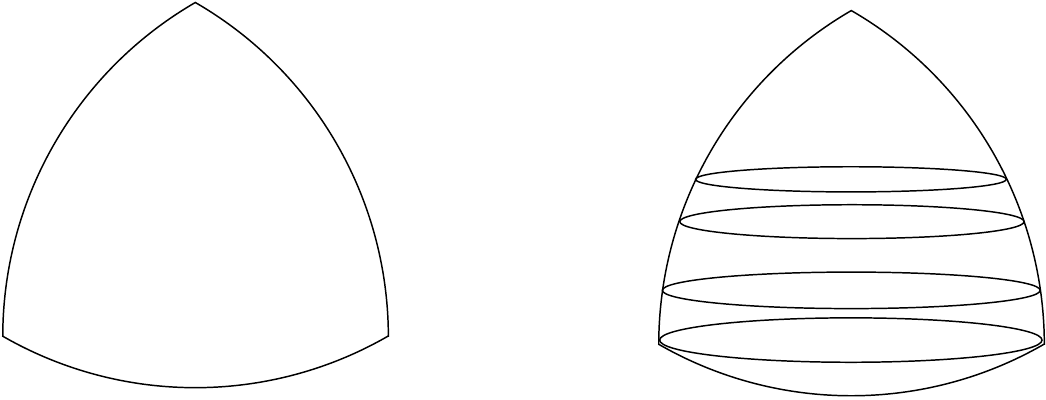}
\caption[]{The body of revolution with minimal volume.}
\label{fig:reuleaux}
\end{figure}

A candidate to solve Problem B was proposed by Heil \cite{heil}; in this case the construction is based upon a tetrahedron, too. Namely the Heil body is the convex hull of six circular arcs of radius $d$, centered at the mid-points of the edges of a regular tetrahedron of edge's length $d\sqrt{2}$, and the four vertices of a rescaled tetrahedron of edge's length $d(2\sqrt{6}-\sqrt{2})/3$.

\begin{figure}
\centering
\includegraphics[scale=0.6]{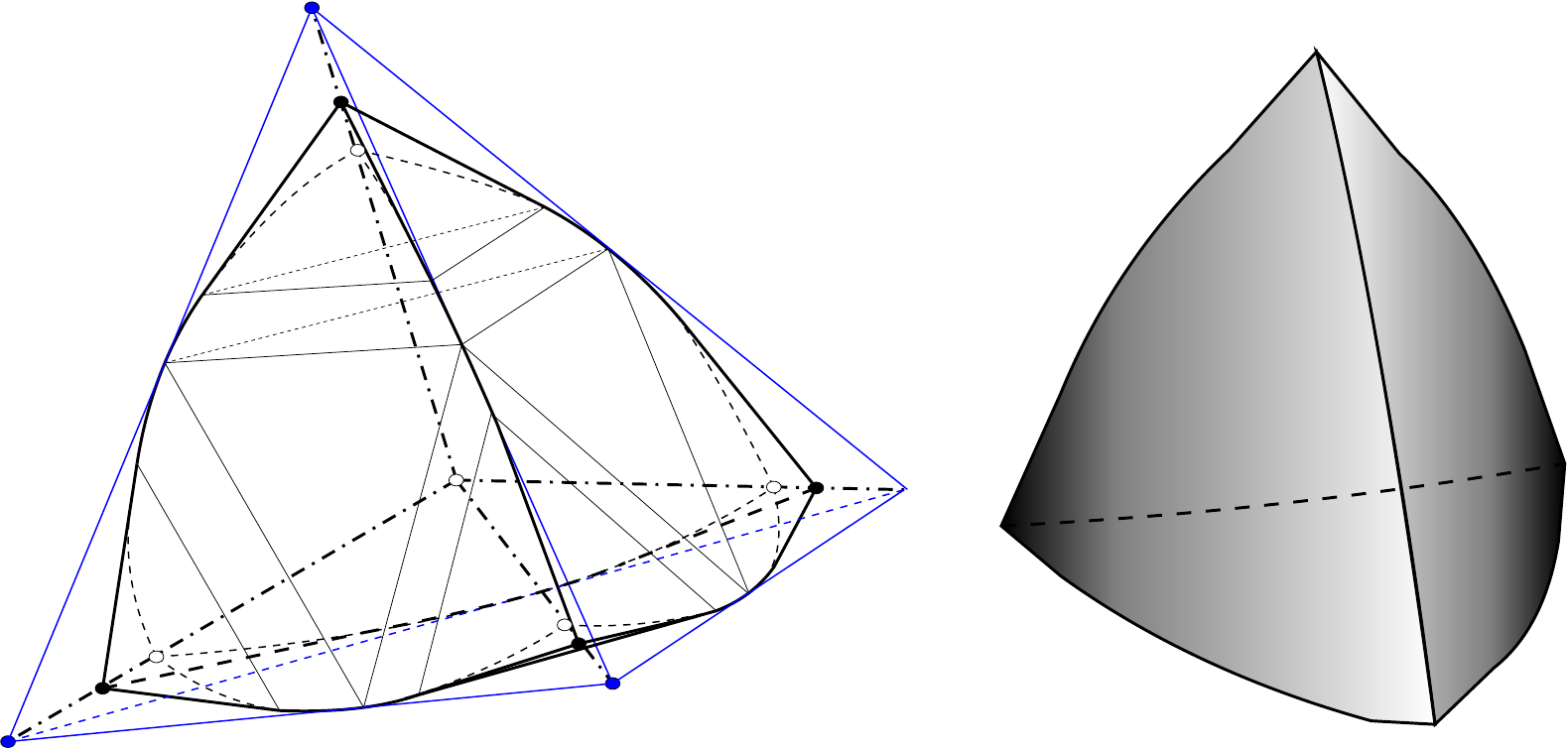}
\caption[]{The Heil-body.}
\label{fig:heilbody}
\end{figure}

\begin{conj}[Heil, 1978]\label{conj:heil}
Heil body is the minimum-volume three-dimensional body in the class of convex bodies with given thickness.
\end{conj}

For surface of revolution Campi, Colesanti and Gronchi solved this problem using the concept of shaken body. Fix a plane $H$ orthogonal to $u\in \S^{n-1}$ and a closed half-space $H^+$ bounded by $H$. The \emph{shaken body} $S_K(u)$ of $K$ with respect to $u$, is the set contained in $H^+$ such that for every line $l$ parallel to $u$, $S_K(u)\cap l$ is either a segment having an endpoint on $H$ and the same length of $K\cap l$, or the empty set, whether $l$ intersects $K$ or not (see \cite{bonnesen-fenchel}). Clearly, the shaken body is a convex body with the same volume as of $K$. On the other hand, Campi at all proved the  width function of the shaken body in the direction of the axis of a three-dimensional convex body of revolution is greater than the width function of $K$. From this the author shew

\begin{theorem}[Campi-Colesanti-Gronchi, 1996]\label{thm:c-c-ggiventhickness}
Among all three-dimensional convex bodies of revolution with given thickness $d$, the unique body of minimum volume is the cone generated by the revolution of a regular triangle of side $2d/\sqrt{3}$, around one of its axes of symmetry.
\end{theorem}

\subsection{The brightness function, the projection body and the Blaschke body}

\begin{figure}
\centering
\includegraphics[scale=0.8]{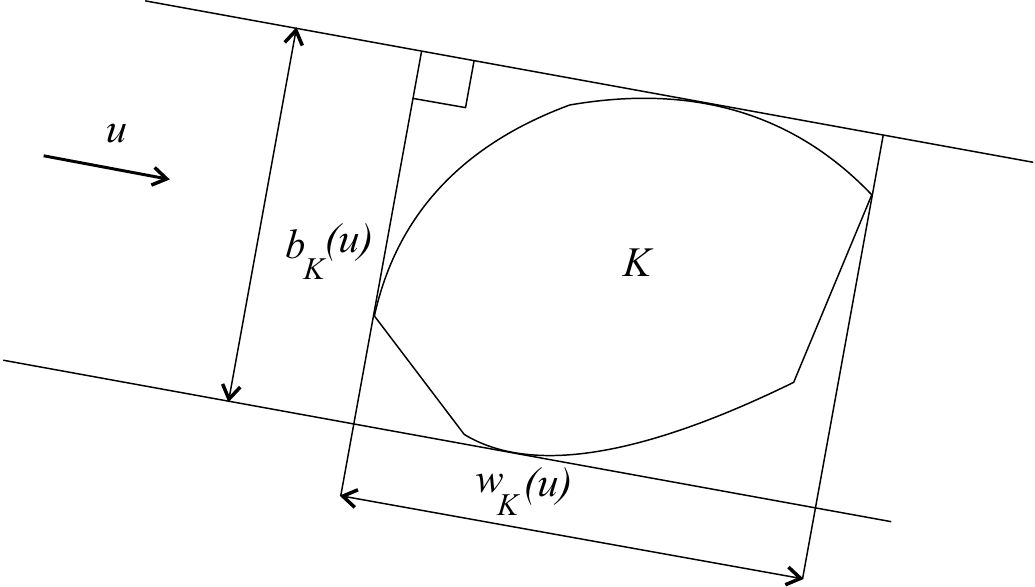}
\caption{The width function and the brightness function of a disk}
\label{fig:brightnessintheplane}
\end{figure}
The \emph{brightness function} of $K$ is also determined by the covariagram function. For a disk it can be get as the rotation of the width function around the $z$-axis. The \emph{ brightness} of $K$ \emph{in the direction of the unit vector} $u$ is the $(n-1)$-dimensional volume of the orthogonal projection of $K$ to a hyperplane with unit normal vector $u$. The \emph{brigthness function} is the function
\begin{equation}\label{def:brightnessfunction}
b_K(u):u\mapsto \vol_{n-1}(K|u^\perp),
\end{equation}
where $(K|u^\perp)$ is the orthogonal projection of $K$ onto the hyperplane with normal vector $u$. Our notes on the brigthness function of a disk is clear from Figure \ref{fig:brightnessintheplane}. The covariagram function determines the brightness function of the body. In fact, for $u \in \S^{n-1}$,
\begin{equation}\label{eq:derivofcov}
\lim\limits_{r\mapsto 0+0}\frac{\mathrm{d}}{\mathrm{d}r}g_K(ru)=-\vol_{n-1}(K|u^\perp)=-b_K(u)
\end{equation}
This follows from the facts that
$$
r\vol_{n-1}((K \cap (K + ru)) | u^{\perp}) \leq \vol_n(K \setminus (K + ru)) \leq r\vol_{n-1}(K | u^{\perp})
$$
and $\lim_{r\mapsto 0+0} K\cap (K + ru) = K$. The answer for the question whether the brightness function determines or doesn't determine the body is known. Let's first introduce the \emph{projection body} $\Pi K$ of $K$ as the uniquely defined body which support function at the point $u$ is equal to the value of the brightness function of $K$ at $u$. If we restrict our investigations to centred convex bodies of $\mathbb{R}^n$ then from Alexandrov's projection theorem (see in \cite{gardner}, Theorem 3.3.6) we get that if $K_1$ and $K_2$ have the same brightness function (or equivalently the projection bodies $\Pi K_1$ and $\Pi K_2$ are agree) then $K_2$ is a translate of $K_1$. In general, as we will see in later, with the same brightness function as $K$, there will be a continuum of (generally non-congruent) sets with this property.

On Figure \ref{fig:projectionbody} we can see the projection body of the regular tetrahedron. To determine the projection body of a polyhedron we can use the Cauchy's projection formula (see A.45 in \cite{gardner}), which connects the volume of the projections by the surface area measure as follows:
\begin{equation}\label{eq:Cauchy'sproj}
h_{\Pi K}(u):=b_K(u)=\vol_{n-1}(K_1|u^\perp)=\frac{1}{2}\int\limits_{\S^{n-1}}|\langle u,v\rangle|\mathrm{d}S_{n-1}(K,v).
\end{equation}

The surface area measure of the regular tetrahedron is the measure concentrated to four vertex of a regular tetrahedron inscribed in the unit sphere. These points have equal measures (see $n_i$ on the figure). The integral then reduces to sum of the four terms $|\langle u, n_i\rangle|$, $i=1,\ldots, 4$. Each term is the support function of a line segment $[0,\vol_{n-1}(F_i)n_i]$, where $F_i$ is the facet orthogonal to $n_i$. By the property of the support function the projection body is the Minkowski sum of the four segments. Hence $\Pi K$ is a rhombic dodecahedron (see on the right of Figure \ref{fig:projectionbody}).

\begin{figure}
\centering
\includegraphics[scale=0.8]{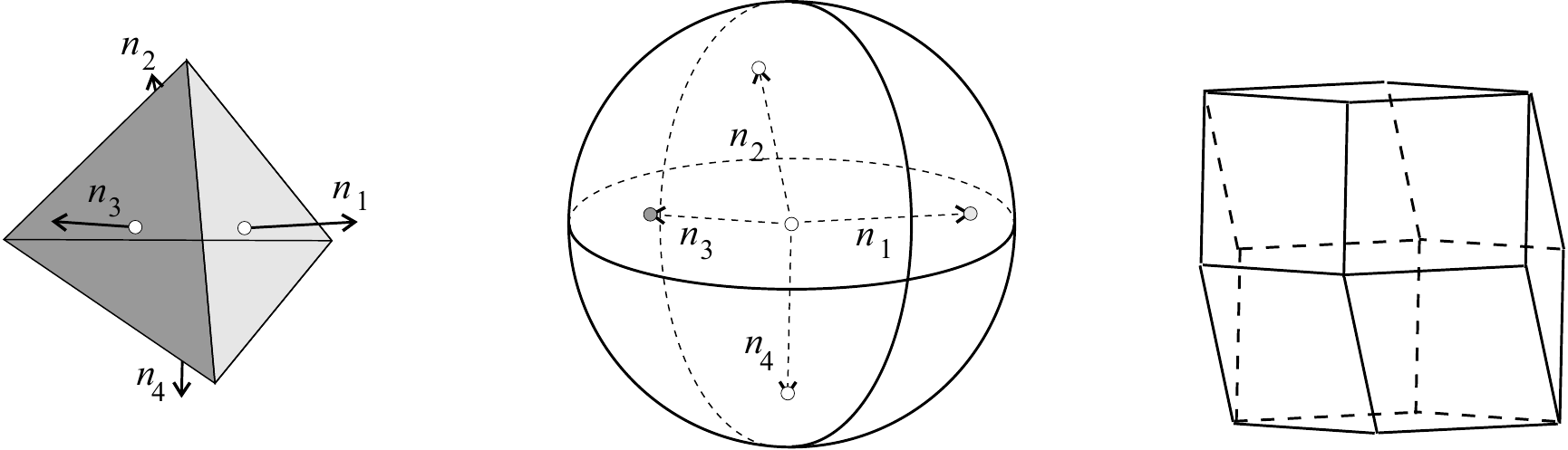}
\caption{The projection body of the regular tetrahedron.}
\label{fig:projectionbody}
\end{figure}

Let $K$ be a convex body of dimension $n$. The \emph{Gauss map} $G$ sends a unit normal vector of the boundary of $K$ to the corresponding point of the sphere $\S^{n-1}$. We can define a measure on $\S^{n-1}$ by the help of the Gauss map. For a set $H\subset \S^{n-1}$ we consider the set of those points $G^{-1}(H)$ of $K$ in which the unit normal vector is mapped by the Gauss map to a point of $H$. Let the measure of $H$ be the surface area of the set $G^{-1}(H)$. We denote by $S_{n-1}(K,\cdot)$ this surface are measure. Minkowki's existence theorem (see A.3.2. in \cite{gardner}) says that for a finite, Borel measure on the unit sphere not be concentrated to a great subsphere of $\S^{n-1}$ there is a convex body $K$ for which the surface area measure $S_{n-1}(K,\cdot)$ is the given one. Hence if $K$ is a convex body in $\mathbb{R}^n$ and $0\leq t\leq 1$, then there is a unique convex body whose surface area measure is $(1-t)S_{n-1}(K,\cdot)+tS_{n-1}(-K, \cdot)$. The connection between the brightness function and the surface area measure of $K$ is the following: The condition that for all $u\in \S^{n-1}$ $\vol_{n-1}(K_1|u^\perp)=\vol_{n-1}(K_2|u^\perp)$ is equivalent to that for all $u\in \S^{n-1}$ $S_{n-1}(K_1,\cdot)+S_{n-1}(-K_1,\cdot)=S_{n-1}(K_2,\cdot)+S_{n-1}(-K_2,\cdot)$ (see Theorem 3.3.2 in \cite{gardner}). This leads to another important observation since for all $0\leq t\leq 1$ holds
$$
((1-t)S_{n-1}(K,\cdot)+tS_{n-1}(-K, \cdot))+((1-t)S_{n-1}(-K,\cdot)+tS_{n-1}(K, \cdot))=
$$
$$
S_{n-1}(K,\cdot)+S_{n-1}(-K,\cdot),
$$
the unique convex body whose surface area measure is $(1-t)S_{n-1}(K,\cdot)+tS_{n-1}(-K, \cdot)$ has the same brightness function as of $K$. This also implies that their projection bodies are the same centred convex body.

The \emph{Blaschke body} $\triangledown K$ corresponds to $t=1/2$. The term \emph{projection class} of $K$ is used for the class of all convex bodies $K'$ for which $\Pi K=\Pi K'$. $\triangledown K$ is the unique body in a projection class with the largest volume. (see Theorem 4.1.3 and Theorem 3.3.9 in \cite{gardner}). From the proof (applying Minkowski's first inequality, B.11 in \cite{gardner}) follows that the volume of $K$ is equal to the volume of $\triangledown K$ if and only if $K$ is centrally symmetric. This means that the only centrally symmetric element of a projection class is the Blaschke body. In Figure \ref{fig:blaschkebody} we drew the Blaschke body of the regular tetrahedron. The surface are measure of $-K$ is the reflected image of the surface are measure of $K$. Hence the surface area measure of the Blaschke body is concentrated to $8$ points with equal masses situated at the outward unit vectors of the facets of a regular octahedron, hence the body $\triangledown K$ is that regular octahedron which is the intersection of the regular tetrahedron $K$ and its reflected image $-K$. (Here the origin is the common centroid of the two tetrahedra.)

\begin{figure}
\centering
\includegraphics[scale=0.8]{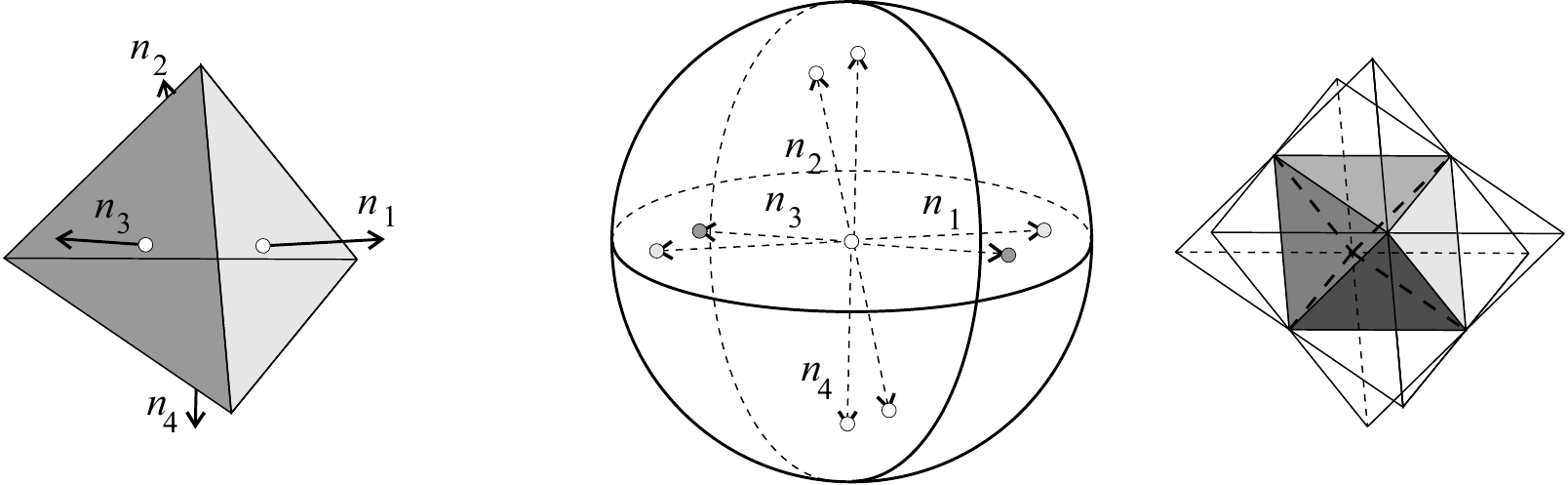}
\caption{The Blaschke body of the regular tetrahedron}
\label{fig:blaschkebody}
\end{figure}

In dimension $3$ a convex body has constant brightness if its shadow on any plane always has the same area. Among centrally symmetric convex bodies the ball is the only with constant brightness (Theorem 3.3.11 in \cite{gardner}). Blaschke constructed the first example of nonspherical convex body in $\R^3$ of constant brightness (see \cite{blaschke-book}). He observed that it can be found such body in the class of solids of revolution with the $z$-axis as axis of revolution. We shortly write this nice construction. It is based on the fact that a convex body of class $C^2_+$ has constant brightness if and only if the sum of the products of the principal radii at antipodal points is constant (see Theorem 3.3.14 in \cite{gardner}). He described the $2$-dimensional meridian section of the body $K$ which lies on the $\{x,z\}$-plane. Let $x_u$ be the point of this meridian where the outer normal vector is $u$. The directions corresponding to the principal radii of curvature are $\{u_1,u_2\}$. We take $u_1$ in the plane of $\{u,z\}$, and orthogonal to $u$. $u_2$ is orthogonal to the plane $\{u,u_1\}$. The second principal radius $R_2$ is the distance from $x_u$ to the $z$-axis measured along the line through $x_u$ parallel to $u$. We need
$$
R_1(u)R_2(u)+R_1(-u)R_2(-u)=\mathrm{const.}=c
$$
for all $u$. We can see in Figure \ref{fig:blaschkeconstbrightness} the examined meridian. Note that the base point of one of two opposite normal vectors is in a vertex of this meridian. Hence the smooth parts have to be of constant curvature with the same value (in a non-smooth point the curvature is infinity giving that one of the principal radiuses is equal to zero). For a surface of revolution of form
$$
F(x,\varphi)=\left(x\cos \varphi, x\sin\varphi, z(x)\right),
$$
the Gaussian curvature $K$ is known it is
$$
\frac{1}{R_1(u)R_2(u)}=K=\frac{\dot{z}\ddot{z}}{x(1+\dot{z}^2)^2}.
$$
By the substitution $t={\cdot z}^2, \dot{t}=2\dot{z}\ddot{z}$ from this in the case of $K\ne 0$ we get the differential equation
$$
\frac{\dot{t}}{(1+t)^2}=2Kx.
$$
The general solution for $t$ is:
$$
\dot{z}^2=t=-\frac{1+c+Kx^2}{c+Kx^2},
$$
where $c$ is an arbitrary constant. In our case of $K=1$ we also know the initial condition: $\dot{z}(0)=-1$. Hence $c=-1/2$, and $z$ is
$$
z(x)=\pm \int\limits_{0}^x \sqrt{\frac{1+2x^2}{1-2x^2}}\mathrm{d}x+C,
$$
with another constant $C$. Substitute $\cos v=\sqrt{2}x$ in this non-elementary integral and we get
$$
z(\frac{\cos v}{\sqrt{2}}=\mp \frac{1}{\sqrt{2}}\int\limits_{\pi/2}^{\cos^{-1}\sqrt{2} x} \sqrt{1+\cos^2v}\mathrm{d}v+C=\pm \frac{1}{\sqrt{2}}\int\limits^{\pi/2}_{\cos^{-1}\sqrt{2} x} \sqrt{2-\sin^2v}\mathrm{d}v+C
$$
The initial condition now is $z(\sqrt{2}/2)=0$, hence
$$
C=\mp \frac{1}{\sqrt{2}}\int\limits^{\pi/2}_{0} \sqrt{2-\sin^2v}\mathrm{d}v.
$$
Since $z(v)\geq 0$ and $\frac{1}{\sqrt{2}}\int\limits^{\pi/2}_{0} \sqrt{2-\sin^2v}\mathrm{d}v\geq \pm \frac{1}{\sqrt{2}}\int\limits^{\pi/2}_{\cos^{-1}\sqrt{2} x} \sqrt{2-\sin^2v}\mathrm{d}v$ we have to choose the positive sign for $C$. Hence the only possibility is
$$
z\left(\frac{\sqrt{2}}{2}\cos v\right)=\frac{\sqrt{2}}{2}\left(\int\limits^{\pi/2}_{0} \sqrt{2-\sin^2v}\mathrm{d}v-\int\limits^{\pi/2}_{\cos^{-1}\sqrt{2} x} \sqrt{2-\sin^2v}\mathrm{d}v\right)=
$$
$$
\frac{\sqrt{2}}{2}\int\limits_{0}^{\cos^{-1}\sqrt{2} x} \sqrt{2-\sin^2v}\mathrm{d}v,
$$
which is the meridian was proposed by Blaschke in \cite{blaschke-book}.
\begin{figure}
\centering
\includegraphics[scale=0.8]{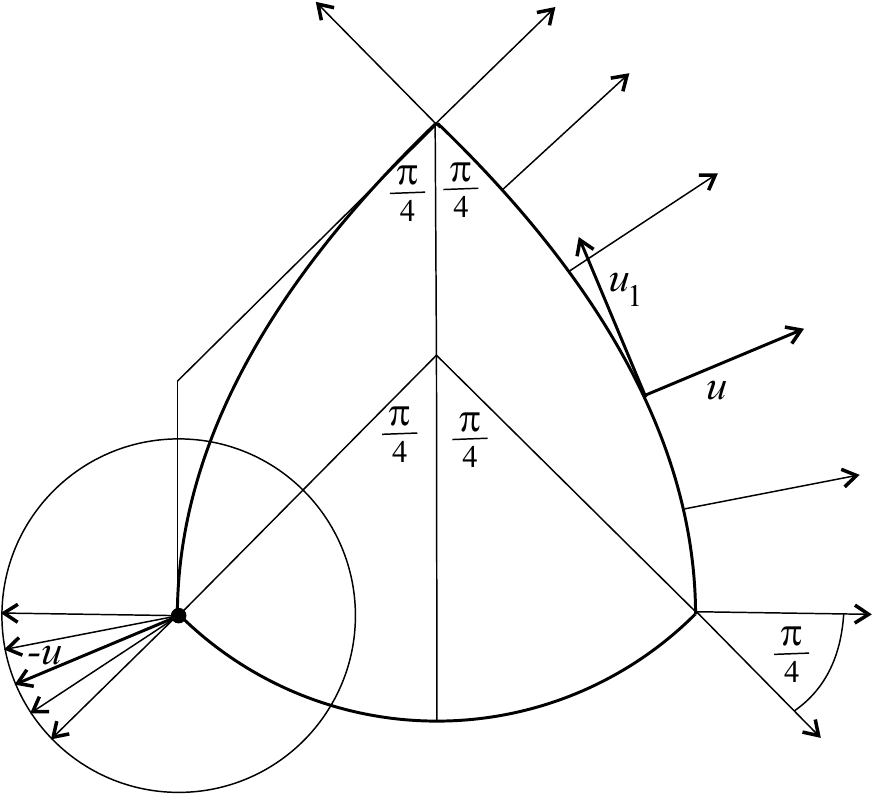}
\caption{A nonspherical body of constant brightness}
\label{fig:blaschkeconstbrightness}
\end{figure}

A nice problem arose in the common investigation of width and brightness. In the plane the two concepts essentially are the same and give the same information on the body. In $3$-space, this is not true so we have the following question: \emph{ Is every convex body in $\R^3$ of constant width and constant brightness is a ball?} The positive answer for $C^2$ bodies was given by Nakajima in 1926 \cite{nakajima}, he showed that any convex body in $\R^3$ with constant width, constant brightness, and boundary of class $C^2$ is a ball. In 2005, Howard proved in \cite{howard} that the regularity assumption on the boundary is unnecessary, so balls are the only convex bodies of constant width and brightness.

Turning back to paper of Campi at all,  we find the following problems:
\begin{problem}[Campi-Colesanti-Gronchi]
Find a convex body of minimum volume in each of the following classes:

C) The class of convex bodies with constant brightness $b$;

D) The class of convex bodies with minimal brightness $b$.
\end{problem}

To get a partial solution of problem C) the authors proved the following

\begin{theorem}[Campi-Colesanti-Gronchi, 1996]\label{thm:c-c-constbrightness}
Let $K$ be a body of minimum volume in the class of all convex bodies having constant brightness $b$. Then the surface area measure $S_2(K,\cdot)$ of $K$ has the following form
\begin{equation}\label{eq:samofminbrightness}
S_2(K,H) = \int\limits_{H} s_K(z)\mathrm{d}z, \quad  H \mbox{ is a Borel set of } \S^2
\end{equation}
where $s_K(\cdot)$ is a non-negative function in $L_1(\S^2)$ such that
\begin{eqnarray}\label{eqn:densityfunc}
  s_K(z)+s_K(-z)& = & -\frac{2b}{\pi} \quad \mbox{almost ewerywhere in} \quad \S^2 \\
  s_K(z)s_K(-z) & = & 0 \quad \mbox{almost ewerywhere in} \quad \S^2 .
\end{eqnarray}
\end{theorem}

This theorem leads to further conjectures:

\begin{conj}[Campi-Colesanti-Gronchi on Problem C:]
Consider three pairwise orthogonal great circles on $\S^2$, they determine eight open regions. Define $s_K$ as a piecewise constant function whose values are $0$ and $2b/\pi$ on those regions alternately. Let $K$ be the unique convex body such that $s_K$ is the distribution of its surface area measure. Clearly, $K$ has constant brightness $b$, and the Gaussian curvature is $\pi/2b$ at each regular point of $K$.
\end{conj}

\begin{conj}[Campi-Colesanti-Gronchi on Problem C when the body is a body of revolution:]
The body was constructed by Blaschke has minimum volume among all convex bodies of revolution with constant brightness $b$.
\end{conj}

To support this conjecture they noticed that Blaschke's body of constant brightness can be constructed on the following manner: Fix
on $\S^2$ a great circle as an equator and take the pair of parallel circles at a spherical distance $\pi/4$ from it. In such a way $\S^2$ is divided into four open regions. Define $s_K$ to be a piecewise constant function whose values are $0$ and $2b/\pi$ on those regions alternately; $s_K$
is the distribution of the area measure of the convex body of revolution with constant brightness constructed by Blaschke in \cite{blaschke-book}.

We note that problem D) is also open recently, even the regular simplex is a strong candidate to solve it.

\section{The convex-hull function}

The convex-hull function of a convex body $K$ is the "dual" of the covariogram function. It was examined for a long time, too. The first paper in connection with it is written in 1950 due to Fary and R\'edei \cite{fary-redei}. They proved that if one of the bodies moves on a line with constant velocity then the volume of the convex hull is a convex function of the time (see Satz.4 in \cite{fary-redei}).  The definition in general is the following:

\begin{defi}
Let $K$ be an $n$-dimensional convex compact body and for a translation vector $t\in \mathbb{R}^n$ associate the value ${G}_K(t)=\vol\mathrm{conv}\{K\cup (K+t)\}$. We call this function the \emph{convex-hull function}, associated to the body $K$.
\end{defi}

In dimension $2$ this statement says that an intersection of the graph of the convex-hull function with a plane is parallel to the $z$-axis is a convex function. Later several paper connected to this observation giving a lot of interesting problem (see the survey \cite{gho-surveyonconvhullvolume} and the references therein). It is an interesting fact, that the "basic problem" on the convex-hull function whether \emph{the convex-hull function determines or doesn't determine the body $K$} was asked only in the close past by \'A. Kurusa.

To answer this question we have to collect some general information on the convex-hull function. Obviously ${G }_K(0)=\vol (K)$ and its support is the whole $\mathbb{R}^n$. Denote by $u$ a unit vector and (as in the previous sections) by $K|u^\perp$ the orthogonal projection of $K$ onto a hyperplane with normal vector $u$.

\begin{lemma}
Let $\alpha$ arbitrary real number and $u\in \S^{n-1}$ unit vector in $\mathbb{R}^n$. Then we have
 \begin{equation}\label{eq:basicid}
 G_K(\alpha u)=\vol_n(K)+|\alpha|\vol_{n-1}(K|u^\perp),
 \end{equation}
consequently the function $u\mapsto G_K(u)-\vol_{n}(K)$ is the brightness function of $K$. Conversely, the brightness function of $K$ and the value $\vol_n K$ determine the convex-hull function of $K$.
\end{lemma}

\begin{proof}
The equality (\ref{eq:basicid}) is an easy consequence of Cavalieri principle (see also in Section A.5 in \cite{gardner}). The second statement follows from the definition of the brightness function. The third statement is an immediate consequence of the equality \ref{eq:basicid}.
\end{proof}

\begin{remark}
In the case, when the brightness function ok $K$ determines the body $K$, the convex-hull function of $K$, too. On the other hand, we discussed earlier that there are several bodies with the same brightness function. The only question is that if the volume and also the brightness function are known that the body is determined or isn't. Since the covariogram function also holds this two properties in dimension $n$ for $n\geq 4$ the answer should be "no".
We know that on the plane there are two convex polygons with the same width function and also have the same volume. To get an example take two congruent regular hexagon and add (in an suitable manner) to three congruent suitable equilateral triangles each of them that the getting nine-sided polygons won't be congruent to each other. (In Figure \ref{fig:samevolumeandsamewidth} we can see the two nine-sided polygons which are non-congruent parts of a regular twelve-sided polygon.) We might observe that the support functions $h_{P}$ and $h_Q$ of $P$ and $Q$, has the property that $\{h_P(u),h_P(-u)\}=\{h_Q(u),h_Q(-u)\}$. Hence $w_P(u)=w_Q(u)$ for all $u$, furthermore the width function is a rotated copy of the brightness function, hence this example holds the two required properties.

\begin{figure}[ht]
\centering
\includegraphics[scale=0.8]{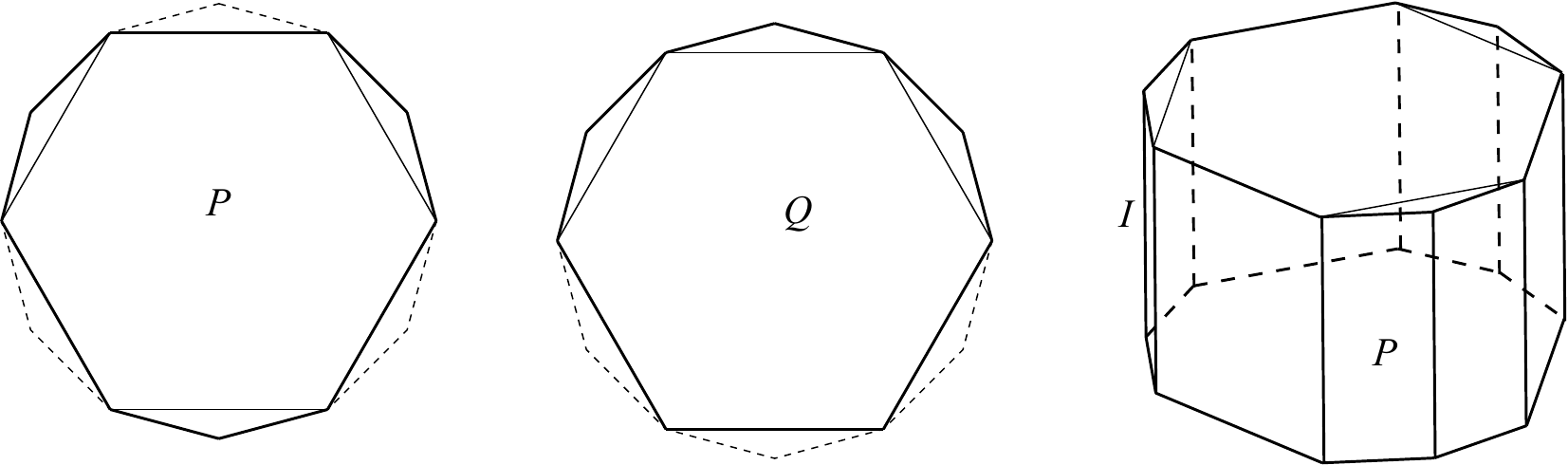}
\caption{Non-congruent polygons with the same brightness function and with the same volume}
\label{fig:samevolumeandsamewidth}
\end{figure}
\end{remark}

From these polygons $P$ and $Q$, we can construct a $3$-dimensional pair of polyhedrons with the same brightness function and with the same volume. Consider the prism $H_P:=P\times I$ and $H_Q:=Q\times I$ where $I$ is a segment orthogonal to the common plane $H$ of $P$ and $Q$. By the following theorem, these prisms have the required properties.

\begin{theorem}\label{prop:samebrightnessandvolume}
Let $P$ and $Q$ be convex bodies of dimension $n$ with the same brightness function and equal volumes. Let $I$ be a segment orthogonal to that subspace $H$ of the $(n+1)$-dimensional Euclidean space which contains $P$ and $Q$. Then the prisms $H_P$ and $H_Q$ also have the same brightness function and also have equal $(n+1)$-dimensional volume.
\end{theorem}

\begin{proof}
The equality $\vol_{n+1}(H_P)=\vol_{n+1}(H_Q)$ is obvious.

First we prove that the orthogonal projections $P'$ and $Q'$ of $P$ and $Q$ to a subspace $\Pi$ with unit normal vector $u$ have the same Hausdorff measure. (In the case, when $\Pi$ doesn't orthogonal to $H$ the projection has non-zero $n$-dimensional volume which agrees with the measure above, and in the orthogonal case this measure is equal to the $(n-1)$-dimensional volume of the projection.) Let the unit normal vector of $H$ is $h$ and $I=\alpha h$. Since $P$ and $Q$ have the same $n$-dimensional volume we have $\vol_n(P')=|\langle u,h\rangle |\vol_n(P)=|\langle u,h\rangle | \vol_n(Q)=\vol_n(Q')$. If it is non-zero then the projections have the same $n$-dimensional volume. If it is zero $H$ and $\Pi$ are orthogonal to each other and the projections of $P'$ and $Q'$ are the respective shadows of $P$ and $Q$ on an $n-1$-dimensional subspace of $H$. Since their brightness functions are agree the $(n-1)$-dimensional volumes of $P'$ and $Q'$ are equals to each other, as we stated.

Assume that $H$ and $\Pi$ are not orthogonal. Then the projection $(H_P|u^\bot)$ is the convex-hull of the projection $P'$ and $P'+(\alpha h-\langle u,\alpha h \rangle u)$ of the polytopes $P$ and $P+\alpha h$ of dimension $n$, where $+$ denotes the vector sum. Hence by Equation  \ref{eq:basicid}
$$
\vol_n(H_P|u^\bot)=\vol_n(P')+|\alpha||( h-\langle u, h \rangle u)|\vol_{n-1}(P'|( h-\langle u, h \rangle u)^\bot).
$$
On the other hand we have
$$
(P'|( h-\langle u, h \rangle u)^\bot)=((P|\Pi)|( h-\langle u, h \rangle u)^\bot)=(P|H\cap \Pi),
$$
since $\langle v,u\rangle=0$ and $ \langle v,h-\langle u, h \rangle u\rangle=0$ ensures $\langle v,h\rangle=0$.
Thus
$$
\vol_n(H_P|u^\bot)=\vol_n(P')+|\alpha||( h-\langle u, h \rangle u)|\vol_{n-1}(P|H\cap \Pi)=
$$
$$
\vol_n(Q')+|\alpha||( h-\langle u, h \rangle u)|\vol_{n-1}(Q|H\cap \Pi)=\vol_n(H_Q|u^\bot),
$$
as we stated.

If $H$ and $\Pi$ are orthogonal then
$$
\vol_n(H_P|u^\bot)=\vol_{n-1}(P|H\cap \Pi)=\vol_{n-1}(Q|H\cap \Pi)=\vol_n(H_Q|u^\bot),
$$
immediately gives the required result.
\end{proof}

Using Proposition \ref{prop:samebrightnessandvolume} we get example in arbitrary dimension for non-congruent convex bodies with the same brightness function and volume.

The following problem leads to some interesting connections among volume functions.

\begin{defi}[\cite{gho-langi-convhull}]\label{def:translconstvol}
If, for a convex body $K \in \R^n$, we have that $\vol_n (\conv ((v+K) \cup (w+K)))$ has the same value for any touching pair of translates,
let us say that $K$ satisfies the \emph{translative constant volume property}.
\end{defi}

If $K$ is centred and satisfies the translative constant volume property then $G_K(x)$ depends only on the norm $\|x\|_K$ of $x$. Hence the question which bodies satisfy the translative constant volume property is analogous to the question of Meyer-Reissner-Schmuckenschl\"ager \cite{meyer-reisner-schmuckenschlager} on covariogram in 1993. They proved that if $K$ a centred body and it has the property that $g_K(x)$ depends only on the norm $\|x\|_K$ of $x$, then $K$ is an ellipsoid.

We recall that a $2$-dimensional $o$-symmetric convex curve is a Radon curve, if, for the convex hull $K$ of a suitable affine image of the curve, it holds that $K^\circ$ is a rotated copy of $K$ by $\frac{\pi}{2}$ (cf. \cite{martini-swanepoel-antinorm}). We noted in \cite{gho-langi-convhull} that the concept of Radon curve arose in connect with the examination of the Birkhoff orthogonality in Minkowski normed spaces.

In the paper \cite{gho-langi-convhull} we can find the following theorem:

\begin{theorem}[G.Horv\'ath-\L\'angi,2014]\label{thm:gholangitransvolprop}
For any disk (plane convex body) the following are equivalent.
	\begin{itemize}
		\item[(1)] $K$ satisfies the translative constant volume property.
		\item[(2)] The boundary of the central symmetral of $K$ is a Radon curve.
		\item[(3)] $K$ is a body of constant width in a Radon norm.
	\end{itemize}
\end{theorem}
This statement motivates a new conjecture, since it is known (cf. \cite{alonso-benitez} or \cite{martini-swanepoel-antinorm}) that for $d \geq 3$, if every planar section of a normed space is Radon, then the space is Euclidean; that is, its unit ball is an ellipsoid.

\begin{conj}[G.Horv\'ath-L\'angi, 2014]\label{conj:translconstvol}
Let $d \geq 3$. If some centrally symmetric convex body $K \in \R^d$ satisfies the translative constant volume property, then $K$ is an ellipsoid.
\end{conj}

If $K$ and $K+t_p$ are touching in the point $p\in K$ then $p-t_p\in K$ also holds and from Equation \ref{eq:basicid} we get that
$G_K(t_p)=\vol(K)+|t_p|\vol_{n-1}(K|t_p^\perp)$. The chord $[p,p-t_p]$ of $K$ is a so-called \emph{affine diameter} of $K$. The general properties of affine diameters are collected in the survey of V. Soltan in \cite{soltan}. In the centrally symmetric case we know that a chord containing the centre is an affine diameter. These affine diameters determine the \emph{radial function} of the body $K$. If the origin $x$ is the centre of $K$ then we define the radial function with the equality:
\begin{equation}\label{def:radialfunction}
\rho_K(u):=\sup\{t\in\R: tu\in K\}.
\end{equation}
The radial function can be considered for all convex body if we assume that the origin is an interior point of $K$. There is the following connection between the support function and the radial function:
$$
\rho_K(u)=\frac{1}{h_{K^\circ}(u)},
$$
where $K^\circ$ is the polar body of $K$. The \emph{polar projection body} of $K$ is the body $\Pi^\circ K:=\Pi K^\circ$. An old problem is the so-called \emph{polar projection problem} is raised by Petty in \cite{petty} and is mentioned by Gruber in \cite{gruber} and Lutwak in \cite{lutwak-affineisop}.

\begin{conj}[Petty, 1971]\label{conj:polarprojection}
If some centrally symmetric convex body $K \in \R^d$ satisfies $\Pi^\circ K=\lambda K$, then $K$ is an ellipsoid.
\end{conj}

This problem has a form also in the non-symmetric case, asking that which bodies has the property that its projection bodies and difference bodies are polars of each other. Martini in \cite{martini-polarprojection} proved that for polytopes the only one with this property is the simplex. The reason that we mentioned here the polar projection problem that it is equivalent to the conjecture on translative constant volume property. (We will prove it in a forthcoming paper.)

\section{Remarks on the homothetic versions of the above problems}

First of all, I would like to mention the paper of Meyer-Reisner-Schmuckenschl\"ager \cite{meyer-reisner-schmuckenschlager} in which we can find the following theorem:

\begin{theorem}[Meyer-Reisner-Schmuckenschl\"ager,1993] Let $K\in \R^n$ be a centrally symmetric convex body. If for some $\tau$ the
volume $\vol_n (K \cap \{\tau K + x\})$ depends only on the Minkowski norm $\|x\|_K$, then $K$ is an ellipsoid.
\end{theorem}

For the convex-hull function an analogous was given by J.J.Castro in \cite{castro}. He proved the following two statements:

\begin{theorem}[Castro, 2015]Let $K\in \R^n$ be a convex body with the origin $O$ in its interior. If there is
a number $\lambda \in (0, 1)$ such that $\vol_n \conv\{(K \cup \lambda K + x)\}$ depends only on the Euclidean norm
$\|x\|$, then K is a Euclidean ball.
\end{theorem}

\begin{theorem}[Castro, 2015]Let $K\in \R^n$ be a convex body with the origin O in its interior and let
$L \subset \R^n$ be a centrally symmetric convex body centered at the origin. If there is a
number $\lambda \in (0, 1)$ such that $\vol_n \conv\{(K \cup \lambda K + x)\}$ depends only on the Minkowski norm
$\|x\|_L$ , then $L$ is homothetic to $K$.
\end{theorem}

Unfortunately, the witty proofs of these statements cannot be applied to the case of Conjecture \ref{conj:translconstvol}. However, these results suggest the following problem:

\begin{defi}
Denote by
$$
G_{K,\lambda}(t):=\vol_n\conv\left\{K\cup \left(\lambda K+t\right)\right\}
$$
the \emph{$\lambda$-homothetic convex-hull function} of $K$.
\end{defi}

\begin{problem}\label{prob:homconvhull}
Does the $\lambda$-homothetic convex-hull function $G_{K,\lambda}(t)$ determine the body $K$? What we can say on $K$ if we know more $\lambda$-homothetic convex-hull functions of it?
\end{problem}

\end{document}